\documentclass{article}

\usepackage{amssymb}
\usepackage{amsmath}
\usepackage{amsthm}

\newtheorem{theorem}{Theorem}[section]
\newtheorem{definition}{Definition}[section]
\newtheorem{lemma}{Lemma}[section]
\newtheorem{corollary}{Corollary}[section]
\newtheorem{conjecture}{Conjecture}[section]

\title{Multi-latin squares}

\author{
Nicholas Cavenagh \\
Department of Mathematics \\
The University of Waikato \\
Private Bag 3105, Hamilton, New Zealand \\
{\texttt{nickc@maths.waikato.ac.nz}} \\
~ \\
Carlo H\"{a}m\"{a}l\"{a}inen \footnote{H\"{a}m\"{a}l\"{a}inen supported by the Eduard Cech Center under the grant LC505.}\\
Department of Mathematics \\
Charles University \\
Sokolovsk\'a 83, 186 75 Praha 8 \\
Czech Republic \\ 
{\texttt{carlo.hamalainen@gmail.com}} \\
~\\
James G. Lefevre\footnote{Lefevre supported by grants DP0664030 and LX0453416.}\\
Centre for Discrete Mathematics and Computing \\
University of Queensland \\
St Lucia 4072 Australia \\
{\texttt{jgl@maths.uq.edu.au}} \\
~\\
Douglas S. Stones\footnote{Supported by ARC Grant DP0662946}\\
School of Mathematical Sciences \\
Monash University \\
Vic 3800 Australia
}

\date{}

\begin{document}

\maketitle\thispagestyle{empty}
\def\baselinestretch{1.0}\small\normalsize
\def \l {\lambda}
\def \mod {\mbox{ mod }}
\parskip=2mm

\begin{abstract}
A multi-latin square of order $n$ and index $k$ is an $n\times n$
array of multisets, each of cardinality $k$, such that each symbol from a
fixed set of size $n$ occurs $k$ times in each row and $k$ times in
each column.  A multi-latin square of index $k$ is also referred to
as a $k$-latin square.  A $1$-latin square is equivalent to a latin
square, so a multi-latin square can be thought of as a generalization
of a latin square.

In this note we show that any partially filled-in $k$-latin square of
order $m$ embeds in a $k$-latin square of order $n$, for each $n\geq
2m$, thus generalizing Evans' Theorem.  Exploiting this result, we
show that there exist non-separable $k$-latin squares of order $n$
for each $n\geq k+2$.
We also show that for each $n\geq 1$, there exists some finite 
value
$g(n)$ such that for all $k\geq g(n)$, every $k$-latin square of order $n$ is separable. 

We discuss the connection between $k$-latin squares and related
combinatorial objects such as orthogonal arrays, latin parallelepipeds,
semi-latin squares and $k$-latin trades.  We also enumerate and
classify $k$-latin squares of small orders.

\vspace{2mm}
\end{abstract}

\section{Introduction}\label{sec1}

For each positive integer $a$, we use the notation $N(a)$ for the
set of positive integers $\{1,2,\dots ,a\}$.  The operations and relations 
used in this paper always take into account the multiplicity of
symbols in multisets.  For example, if $A$ and $B$ are multisets,
and $A$ contains $t_1$ occurrences of symbol $x$ and $B$ contains
$t_2$ occurrences of symbol $x$, then $A\cup B$ contains $t_1+t_2$
occurrences of symbol $x$.  Similarly, $A\subseteq B$ if and only
if for each symbol $x\in A$ that occurs $t_1$ times in $A$ and $t_2$
times in $B$, $t_1\leq t_2$.

A {\em partial $k$-latin square} of order $n$ is an $n\times n$
array, where each cell of the array contains a multiset of cardinality at most $k$ 
with symbols from $N(n)$, such that each symbol occurs at
most $k$ times in each row and at most $k$ times in each column.
A {\em $k$-latin square} is a partial $k$-latin square in which
each cell contains exactly $k$ symbols, and hence each symbol
occurs precisely $k$ times in each row and $k$ times in each column.
We sometimes refer to (partial) $k$-latin squares as {\em (partial)
multi-latin squares}, of {\em index} $k$.

Below is a $2$-latin square of order $4$:
$$\begin{array}{|c|c|c|c|}
\hline
1,2 & 1,2 & 3,3 & 4,4 \\
\hline
2,3 & 2,3 & 4,4 & 1,1 \\
\hline
1,4 & 3,4 & 1,2 & 2,3 \\
\hline
3,4 & 1,4 & 1,2 & 2,3 \\
\hline
\end{array}$$

Note that a (partial) $1$-latin square is equivalent to a (partial) latin
square. Thus multi-latin squares may be thought of as generalizations of latin squares. 

A (partial) $k$-latin square is said to be {\em simple} if each cell
contains no repeated symbols. So multisets are forbidden in simple partial 
$k$-latin squares.

For a (partial) $k$-latin square $L$ of order $n$, and for any $i,j\in
N(n)$, write $L(i,j)=A$ whenever cell $(i,j)$ of $L$ contains
the multiset $A$.  We may thus consider a (partial) $k$-latin square
as a set of ordered triples of the form $(i,j,L(i,j))$, where in some
cases $L(i,j)$ may be equal to the empty set.

In this sense, a (partial) $k$-latin square $L$ of order $n$ can be
thought of as a subset of a (partial) $k$-latin square $L'$ of order
$n$ if and only if for each $i,j\in N(n)$, $L(i,j)\subseteq L'(i,j)$.

We begin with some straightforward existence lemmas.
\begin{lemma}
For all positive integers $n$ and $k$, there exists a $k$-latin square 
of order $n$.
\end{lemma}

\begin{proof}
Let $L$ be a latin square of order $n$. 
Let $L_k$ be the $k$-latin square of order $n$, 
 where 
for each $i,j\in N(n)$,
$L_k(i,j)$ is the multiset consisting of $k$ copies of 
the symbol in cell $(i,j)$ of $L$.
\end{proof}

Throughout this paper, for any integer $x$ and positive integer $n$,
we define $(x \mod n)$ to be the unique member of $N(n)$ which is
congruent to $x$ modulo $n$.

\begin{lemma}
For all positive integers $n$ and $k$, there exists a simple $k$-latin square 
of order $n$ if and only if $n\geq k$.
\end{lemma}

\begin{proof}
If $n<k$, then we are forced to have at least one repeated symbol in a cell, contradicting 
the simple criterion. For each $n\geq k$,
we define a simple $k$-latin square $L_k$ of order $n$ as follows. 
For each $i,j\in N(n)$,
\[
L_k(i,j)= 
\{ (i+j \mod n), (i+j+1 \mod n), 
  \dots , (i+j+k-1 \mod n) \}.\]
\end{proof}

We explore the relationship between multi-latin squares and other
combinatorial configurations in Section~\ref{sec2}.  In Section~\ref{sec3}, we examine
whether standard embedding theorems on latin squares generalize to
multi-latin squares.  In Section~\ref{sec4}, we explore when multi-latin
squares ``separate'' into other multi-latin squares.  Making use
of the embedding results from Section~\ref{sec3}, we show the existence of
non-separable $k$-latin squares of order $n$ for each $n\geq k+2$.
Then in Section~\ref{sec5} we count and classify $k$-latin squares of order $n$
for small values of $k$ and $n$.
Finally in Section 6 we mention some possible applications of multi-latin squares to experimental design. 

\section{Equivalences and Connections}\label{sec2}

Multi-latin squares are equivalent to at least three other types of combinatorial objects,
and are related to many more. 

There is an equivalence between multi-latin squares and certain 
 {\em orthogonal arrays}. 
\begin{definition}
An {\em orthogonal array} of {\em size} $N$, with $m$ constraints, $q$
levels, strength $t$, and index $\lambda$, denoted by
$\textnormal{OA}_{\lambda}(N, m, q, t)$ is an $m \times N$ array with entries
from $\mathbb{Z}_q$, such that in every $t \times N$ subarray, every
tuple in $\mathbb{Z}_q^t$ appears exactly $\lambda = N/q^t$ times.
\end{definition}
Observe that a $k$-latin square of order $n$ is equivalent to an orthogonal array
$\textnormal{OA}_k(kn^2, 3, n, 2)$.

The research on orthogonal arrays tends to deal with the existence 
of orthogonal arrays with particular parameters and properties \cite{GRC}. 
As we saw in the previous section, it is trivial to verify 
the existence of $k$-latin squares of order $n$ and hence 
the existence of an
$\textnormal{OA}_k(kn^2, 3, n, 2)$, for each positive $k$ and $n$.

We can also define multi-latin squares in terms of graph decompositions. 
It is well-known that a latin square is equivalent to the decomposition of the edges of a complete tripartite graph $K_{n,n,n}$ into triangles.
This equivalence has a natural extension to multi-latin squares.
We define $kK_{n,n,n}$ to be the multi-graph obtained by replacing each edge of $K_{n,n,n}$ with $k$ parallel edges. 
Then 
a $k$-latin square of order $n$ is equivalent to a decomposition of the edges of
$kK_{n,n,n}$ into triangles. 

From this equivalence, it is immediate that any graph automorphism of 
$K_{n,n,n}$ may be applied to a $k$-latin square $L_1$ of order $n$ to
obtain another $k$-latin square $L_2$ of order $n$ which is combinatorially equivalent to $L_1$. 
We say that $L_1$ and $L_2$ with these properties are {\em paratopic} or belong to the same {\em main class} or {\em species}. This extends, in a natural way, existing definitions which are used to describe latin squares. 

Anderson and Hilton \cite{AnHi} define an ({\em exact}) $(p,q,x)$-latin rectangle to be 
a rectangular matrix with $x$ symbols in each cell such that each symbol occurs at most
({\em exactly}) $p$ times in each row and at most ({\em exactly}) $q$ times in each column. 
Thus a $k$-latin square is equivalent to an exact $(k,k,k)$-latin rectangle. Some embedding results for $(p,q,x)$-latin rectangles are given in \cite{AnHi}. 

Next we discuss non-equivalent but related objects to 
multi-latin squares.
A 
$(n\times n)/k$ 
{\em semi-latin} square is 
an $n\times n$ array $A$, whose entries are $k$-subsets of
a $N(kn)$, such that each element of $N(kn)$ occurs exactly once in each row and once in each column of $A$. (These objects are referred to as 
$n\times n$ $r$-multi latin squares in \cite{KuDe,KuDe2}). 
Semi-latin squares have been studied extensively, particularly in terms 
of experimental design. A summary of results on semi-latin squares may be found at \cite{Bailey2}; a list of enumerative results in \cite{BaCh}.  
An embedding theorem for semi-latin squares is given in \cite{KuDe2}.
If no pair of letters occurs more than once in each cell,
the  semi-latin square is called {\em simple}, and is equivalent to a 
SOMA or {\em Simple Orthogonal Multi-Array} (see \cite{Arhin}, \cite{Soi}, \cite{Soi2}). 

Clearly a $(n\times n)/k$ semi-latin square is a distinct concept to a $k$-latin square of order $n$. 
However there are some connections between these two combinatorial objects.
Let $f$ be a function $f:N(kn)\rightarrow N(n)$ such that 
each element of $N(n)$ has a pre-image of size $k$.  
Then, $f$, when applied to the symbol set, maps any 
$(n\times n)/k$ semi-latin square 
to a $k$-latin square of order $n$.

In fact, there is a reverse process. 
Let $L$ be a $k$-latin square of order $n$. 
Given a symbol $x\in N(n)$, construct
a bipartite multi-graph $B_x$ with $t$ edges between
vertices $r_i$ and $c_j$ if and only if $x$ occurs precisely
$t$ times in cell $L(i,j)$. 
Evans' theorem and the rules of a $k$-latin square guarantee that the edges of $B_x$ partition into $k$ pairwise-disjoint perfect matchings.
Thus, given any inverse $f^{-1}:N(n)\rightarrow N(kn)$  
to the function $f$ above, we obtain a $(n\times n)/k$ semi-latin square.  
(Of course $f^{-1}$ is not a well-defined function.)

However, this relation between multi-latin squares and semi-latin
squares is not, in general, one of correspondence.
For example, consider the following $(3\times 3)/2$ semi-latin square:
$$
\begin{array}{|c|c|c|}
\hline
1,2 & 3,4 & 5,6 \\
\hline
3,5 & 6,1 & 4,2 \\
\hline
4,6 & 5,2 & 3,1 \\
\hline
\end{array}
$$
The maps $f_1,f_2:N(6)\rightarrow N(3)$ defined by 
$f_1(1)=f_1(2)=1$,  
$f_1(3)=f_1(4)=2$,  
$f_1(5)=f_1(6)=3$,  
$f_2(1)=f_2(6)=1$,  
$f_2(2)=f_2(3)=2$,  
$f_2(4)=f_2(5)=3$ give rise to 
$2$-latin squares of order $3$, $L$ and $L'$:
  
$$
\begin{array}{|c|c|c|}
\hline
1,1 & 2,2 & 3,3 \\
\hline
2,3 & 1,3 & 1,2 \\
\hline
2,3 & 1,3 & 1,2 \\
\hline
\multicolumn{3}{c}{L} 
\end{array}
\quad
\begin{array}{|c|c|c|}
\hline
1,2 & 2,3 & 3,1 \\
\hline
2,3 & 1,1 & 2,3 \\
\hline
3,1 & 2,3 & 1,2 \\
\hline
\multicolumn{3}{c}{L'} 
\end{array}$$
The $2$-latin squares $L$ and $L'$ are  
not paratopic. 

Just as the differences between latin squares are defined by {\em latin trades} \cite{Cavenagh}, the differences between multi-latin squares gives rise to a type of combinatorial trade. Such trades would include the $t-(v,k)$ latin trades recently introduced in \cite{KhMaMo} and \cite{MaNa}.

It would be an intriguing line of research to further explore the relationship between multi-latin squares, semi-latin squares and SOMAs.

\section{Embeddings}\label{sec3}

For $1\leq m\leq n$, we define an {\em $m\times n$ $k$-latin rectangle} to be a partial $k$-latin square of order $n$, of which 
$m$ rows are filled and the remaining rows are empty.
Evans' theorem tells us that any partial latin rectangle may be extended to a latin square \cite{Ev}; this result can be generalized to $k$-latin rectangles as follows.

\begin{lemma}
Any $m\times n$ $k$-latin rectangle embeds in a $k$-latin
square of order $n$.
\label{fillrows}
\end{lemma}

\begin{proof}
Let $P$ be an $m\times n$ $k$-latin rectangle.
For each $i\in N(n)$, let $D_0(i)$ be the multiset containing $k-z_i(e)$ copies of symbol $e$, $1\leq e\leq n$, where
$z_i(e)$ is the number of occurrences of $e$ in column $i$ of $P$.
Observe that each $D_0(i)$ has size $k(n-m)$, and each symbol $e$ occurs $k(n-m)$ times in sets of the form $D_0(i)$. 

We now iteratively define multisets 
$D_x(1),\dots ,D_x(n)$, for each $x$ such that $1\leq x <k(n-m)$. 
We assume that for a given $p$, where $0\le p < k(n-m)$, we can choose a symbol from each of the multisets 
$D_p(1),\dots ,D_p(n)$, so that the $n$ symbols chosen are all distinct.
We then remove the selected symbols to obtain $D_{p+1}(1),\dots ,D_{p+1}(n-m)$.
For each $i\in N(n)$, we add the selected symbol from $D_p(i)$ to cell 
$(m+1+\lfloor p/k \rfloor,i)$ of $P$; in this way we obtain a completion of $P$ to a $k$-latin square.

We now justify our inductive assumption.  
By Hall's Theorem, it suffices to show that the union of any $s$-subset of 
$\{D_p(1),\dots ,D_p(n)\}$ contains at least 
$s$ distinct symbols.
Recall that each of the symbols
from $N(n)$ occurs a total of $k(n-m)$ times in the multisets 
$\{D_0(1),\dots ,D_0(n)\}$ 
and that each of these multisets has cardinality $k(n-m)$.
At each stage of the iteration we remove one copy of each of the $n$ symbols (one element from each multiset), 
hence by induction each of the $n$ symbols will occur
a total of $k(n-m)-p$ times in the multisets 
$\{D_p(1),\dots ,D_p(n)\}$  
and each of these
multisets will have cardinality $k(n-m)-p$.
Thus any $s$-subset of these multisets will contain a total of $s(k(n-m)-p)$ elements, with each symbol
occurring at most $k(n-m)-p$ times, and so must contain at least $s$ distinct symbols.
\end{proof}

If we restrict ourselves to simple $k$-latin squares,  
we do not have a direct equivalent of the above lemma. 
For example, the following $2\times 3$ $2$-latin rectangle has no completion to a simple $2$-latin square of order $3$. 
\[
\begin{array}{|c|c|c|}
\hline
1,2 & 1,3 & 2,3 \\
\hline
1,2 & 1,3 & 2,3 \\
\hline
\end{array}
\]
It remains an open problem to determine under what conditions a simple multi-latin rectangle can be extended to 
a simple multi-latin square. 

\begin{theorem}
Any partial $k$-latin square $P$ of order $m$ embeds in a $k$-latin
square of order $n$,
for each $n\geq 2m$.
\label{embed1}
\end{theorem}

\begin{proof}
Let $P_0$ be the partial $k$-latin square of order $n$ defined by
$P_0(i,j)=\emptyset$, if $i>m$ or
$j>m$, and otherwise
\begin{eqnarray*}
P_0(i,j) &=& P(i,j)\cup 
\{ (i+j \mod m)+m, (i+j+1 \mod m)+m, \\
 && \dots , (i+j+k-|P(i,j)|-1 \mod m)+m \}.
\end{eqnarray*}

We now wish to complete rows $1$ through to $m$.
For each $i\in N(m)$, let $X(i)$ be the multiset
$\bigcup_{j=1}^m P_0(i,j)$.
Define $A_0(i)$ to be the multiset containing $k-x_i(e)$ copies of symbol
$e$, $1\le e\le n$,
where $x_i(e)$ is the number of occurrences of $e$ in $X(i)$.
Define $B_0(i)$ to be the multiset consisting of $X(i)$ together with $k(n-2m)$ copies of symbol $i$.
We also define $B_0(i)$ for $m+1\le i \le n-m$; in this case we let $B_0(i)$ be the multiset consisting of 
$k$ copies of each of the symbols $m+1, m+2, \dots , n$.
Note that for $1\le i \le m$ and $1\le j \le n-m$,
$|A_0(i)|=|B_0(j)|=k(n-m)$, and that each of the symbols
from $N(n)$  occurs a total of $k(n-m)$ times in the multisets $A_0(1),\dots ,A_0(m),B_0(1),\dots ,B_0(n-m)$.
The multiset $A_0(i)$ represents the symbols available to complete row
$i$, while $B_0(i)$ helps to ensure completion of the construction but will be discarded.

We now iteratively define a partial $k$-latin square $P_x$ and multisets \newline 
$A_x(1),\dots ,A_x(m),B_x(1),\dots ,B_x(n-m)$, 
for $1\le x \le k(n-m)$.
We assume that for a given $p$, where $0\le p < k(n-m)$, we can choose a symbol from each of the multisets 
$A_p(1),\dots ,A_p(m),B_p(1),\dots ,B_p(n-m)$, so that the $n$ symbols chosen are all distinct.
We then remove the selected symbols to obtain $A_{p+1}(1),\dots ,A_{p+1}(m),B_{p+1}(1),\dots ,B_{p+1}(n-m)$.
For each $i\in N(m)$, We add the selected symbol from $A_p(i)$ to cell 
$(i,m+1+\lfloor p/k \rfloor)$ of $P_p$; in this way we obtain $P_{p+1}$.

We now justify the assumption that for a given $p$, where $0\le p < k(n-m)$, 
we can choose a symbol from each of the multisets \newline
$A_p(1),\dots ,A_p(m),B_p(1),\dots ,B_p(n-m)$, so that the $n$ symbols chosen are all distinct.
By Hall's Theorem, it suffices to show that the union of any $s$-subset of 
$\{A_p(1),\dots ,A_p(m),B_p(1),\dots ,B_p(n-m)\}$ contains at least 
$s$ distinct symbols.
Recall that each of the symbols
from $N(n)$ occurs a total of $k(n-m)$ times in the multisets $A_0(1),\dots ,A_0(m),B_0(1),\dots ,B_0(n-m)$,
and that each of these multisets has cardinality $k(n-m)$.
At each stage of the iteration we remove one copy of each of the $n$ symbols (one element from each multiset), 
hence by induction each of the $n$ symbols will occur
a total of $k(n-m)-p$ times in the multisets $A_p(1),\dots ,A_p(m),B_p(1),\dots ,B_p(n-m)$, and each of these
multisets will have cardinality $k(n-m)-p$.
Thus any $s$-subset of these multisets will contain a total of $s(k(n-m)-p)$ elements, with each symbol
occurring at most $k(n-m)-p$ times, and so must contain at least $s$ distinct symbols.

We have thus shown that the iteration is well defined, and so we obtain a $m\times n$ $k$-latin rectangle $P_{k(n-m)}$ 
which contains $k$ elements in each of the cells in the first $m$ rows, and in which the orginal 
partial $k$-latin square $P$ is embedded.
Finally, to complete rows $m+1$ through to $n$, we simply apply Lemma 
\ref{fillrows}. 
\end{proof}

In fact, the bound in Theorem \ref{embed1} is the best possible. To see this, 
observe that any $k$-latin square of order $m$ cannot be embedded 
in a $k$-latin square of order less than $2m$.
Since our proof of Theorem \ref{embed1} relies on Lemma \ref{fillrows}, 
it does not extend immediately to the case of simple multi-latin squares. 
However we conjecture the following.
\begin{conjecture}
\label{conj1}
For each positive integer $k$ and for each positive integer $m$, 
there exists some finite value $n(m,k)$ such that for any $n\geq n(m,k)$, 
any {\em simple} partial $k$-latin square $P$ of order $m$ embeds in a 
{\em simple} 
$k$-latin
square of order $n$.
\end{conjecture}

\section{Erodability and Separability}\label{sec4}

In this section we explore circumstances under which $k$-latin squares
can be combined or broken down into larger or smaller configurations. 
 
Let $L_1$ be a $k_1$-latin square and $L_2$ be a $k_2$-latin square, each 
of order $n$, for some positive integers $k_1$ and $k_2$.
Then the {\em join} of $L_1$ and $L_2$, denoted by $L_1\oplus L_2$, 
is the $(k_1+k_2)$-latin square $L$ where
$L(i,j)=L_1(i,j)\cup L_2(i,j)$ for each 
$i,j\in N(n)$. 
Clearly the join is a commutative binary operation. 

A $k$-latin square $L$ is said to be {\em separable} 
if there exist $k_1,k_2<k$ such that $k_1+k_2=k$ and
$L$ is the join of a $k_1$-latin square and a $k_2$-latin square;
otherwise it is {\em non-separable}. 
A $k$-latin square is said to be {\em erodable} if it can be
 expressed as the join
of a $(k-1)$-latin square and a latin square; otherwise it is 
{\em non-erodable}. 
Note that if
a $k$-latin square is erodable, then it is separable.

However, the converse is not true.
For example, in Figure 1,   
the $4$-latin square is non-erodable, but can be written as the join of two $2$-latin squares:
\[
\begin{array}{l}
\begin{array}{|c|c|c|}
\hline
1,1,2,2 & 1,1,2,2 & 3,3,3,3 \\
\hline
1,1,3,3 & 2,2,3,3 & 1,1,2,2 \\
\hline
2,2,3,3 & 3,3,1,1, & 1,1,2,2 \\
\hline
\end{array}
\quad 
=
\quad 
\begin{array}{|c|c|c|}
\hline
1,2 & 1,2 & 3,3 \\
\hline
1,3 & 2,3 & 1,2 \\
\hline
2,3 & 3,1 & 1,2 \\
\hline
\end{array}
\quad 
\oplus
\quad
\begin{array}{|c|c|c|}
\hline
1,2 & 1,2 & 3,3 \\
\hline
1,3 & 2,3 & 1,2 \\
\hline
2,3 & 3,1 & 1,2 \\
\hline
\end{array}\\
\\
\centerline{\hbox{\rm Figure 1}}
\end{array}
\] 

In fact, it is clear that for any non-erodable $k$-latin square $L$ 
with $k\geq 2$, 
$L\oplus L$ is separable but non-erodable.

A $k$-latin square is said to be {\em fully separable} 
if it can be written as the join of $k$ latin squares. 

We next explore the relationship between separabaility, latin cubes and 
latin parallelepipeds. 
An $n\times n\times k$ {\em latin parallelepiped}
is a three dimensionsal array $A=[a_{i,j,\ell}]$ where
$1\leq i,j\leq n$, $1\leq \ell \leq k$, $a_{i,j,\ell}\in N(n)$ and $a_{i,j,\ell}\neq a_{i',j',\ell'}$ whenever exactly two of the following conditions hold: 
$i=i'$, $j=j'$, $\ell=\ell'$. 
If $k=n$ we say that $A$ is a {\em latin cube}. 

There is a clear equivalence between 
$n\times n\times k$ {\em latin parallelepipeds} 
and simple, fully separable 
$k$-latin squares of order $n$.  
It is shown in \cite{Koch} (with further examples in \cite{McW1}), 
that not every latin parallelepiped of order $n$ can be extended to a 
latin cube of order $n$. 

It is an intriguing open problem to determine, for each $n$,
 the smallest $k$ such that every $n\times n\times k$ latin parallelepiped 
may be extended to a latin cube of order $n$. 
The best-known constructions come from Kochol:
\begin{theorem}{\rm (\cite{Koch,Koch2})}
For each pair of integers 
$d$ and $n$ such that $d\geq 2$ and $n\geq 2d+1$, 
there exists an $n\times n\times (n-d)$ latin parallepiped 
that cannot be extended to a latin cube of order $n$.
\end{theorem}

By collapsing each 
$n\times n\times (n-d)$ latin parallelepiped from the previous theorem
into a $(n-d)$-latin square of order $n$, then taking the complement (with respect to $N(n)$)
of each set in each cell, we obtain the following corollary.

\begin{corollary}
For all $k$ and $n$ such that $k\geq 2$ and $n\geq 2k+1$, there exists
a $k$-latin square of order $n$ which is {\em not} fully separable.
\end{corollary}  

The $k$-latin squares implied by Kochol's construction are, in general,   
separable and erodable. 
We will now work towards 
showing the existence of non-separable $k$-latin squares
of order $n$, 
where either $k=2$ and $n\geq 3$, or 
when $k\geq 3$ and $n\geq k+2$.

Observe that a $2$-latin square is erodable if and only if it is separable.

\begin{theorem}
There exists a non-separable 
$2$-latin square of order $n$ if and only if $n\geq 3$.
\end{theorem}

\begin{proof}
By inspection, each $2$-latin square of order $1$ or $2$ is separable. 
For each $n\geq 3$, define a $2$-latin square $L_n$ of order $n$ as follows.
$$
\begin{array}{l}
L_n(i,1)=L_n(i,2)=\{i,i+1\}, 1\leq i\leq n-2, \\
L_n(i,j)=\{(i+j-1 \mod n,i+j-1 \mod n)\}, 1\leq i\leq n-2,3\leq j\leq n, \\
L_n(n-1,1)=L_n(n,2)=\{1,n\},  \\
L_n(n,1)=L_n(n-1,2)=\{n-1,n\},  \\
L_n(n-1,j)=L_n(n,j)=\{j-2,j-1\}, 3\leq j\leq n.
\end{array}
$$
(The $2$-latin square $L_3$ is in Figure 1 and $L_4$ is in the Introduction.)
It is not hard to check that $L_n$ is non-separable for each $n\geq 2$. 
It is sufficient to consider the first two columns of $L_n$. 
\end{proof}

Next suppose that $k\geq 3$.
We will first construct a non-separable $k$-latin square of every order from $k+2$ to $2k+1$. 
For ease of understanding we first give the construction for order $k+2$.
Let $(K,\circ)$ be an idempotent quasigroup of order $k$
(where, without loss of generality, $K=N(k)$), 
and for $a,b\in K$ define 
$a\odot b=(a+b$ modulo $k)$.
For ease of notation, we use the abbreviation $x^y$ when $x$ occurs $y$ times in a multiset.
Define $U_k$ to be a $k$-latin square of order $k+2$ with

$$
\begin{array}{l}
U_k(i,i) = \{i,(k+1)^{k-1}\}, \;\; i\in K, \\
U_k( i\odot 1 ,i) = \{k+1, ((i\odot 1 )\circ i)^{k-1}\}, \;\; i\in K, \\
U_k(i,j) = \{k+2, (i\circ j)^{k-1}\}, \;\; i,j\in K, \; i\ne j, j\odot  1, \\
U_k(i,k+1) = U_k(k+1,i) = \{i^{k-1},k+2\}, \;\; i\in K, \\
U_k(i,k+2) = U_k(k+2,i) = \{k+2\}\cup K \setminus \{i\}, \;\; i\in K, \\
U_k(k+1,k+1) = U_k(k+2,k+2) = K, \\
U_k(k+1,k+2) = U_k(k+2,k+1) = (k+1)^k.
\end{array}
$$

We illustrate with an example for the case $k=4$. We use the following idempotent quasigroup
of order $k=4$:

\[
\begin{array}{|c|cccc|}
\hline
\circ & 1 & 2 & 3 & 4 \\
\hline
1 & 1 & 4 & 2 & 3 \\
2 & 3 & 2 & 4 & 1 \\
3 & 4 & 1 & 3 & 2 \\
4 & 2 & 3 & 1 & 4 \\
\hline
\end{array}
\]
We obtain 
\[
U_4=
\begin{array}{|c|c|c|c|c|c|}
\hline
5,5,5,1 & 4,4,4,6 & 2,2,2,6 & 3,3,3,5 & 1,1,1,6 & 2,3,4,6\\
\hline
3,3,3,5 & 5,5,5,2 & 4,4,4,6 & 1,1,1,6 & 2,2,2,6 & 1,3,4,6\\
\hline
4,4,4,6 & 1,1,1,5 & 5,5,5,3 & 2,2,2,6 & 3,3,3,6 & 1,2,4,6\\
\hline
2,2,2,6 & 3,3,3,6 & 1,1,1,5 & 5,5,5,4 & 4,4,4,6 & 1,2,3,6\\
\hline
1,1,1,6 & 2,2,2,6 & 3,3,3,6 & 4,4,4,6 & 1,2,3,4 & 5,5,5,5 \\
\hline
2,3,4,6 & 1,3,4,6 & 1,2,4,6 & 1,2,3,6 & 5,5,5,5 & 1,2,3,4 \\
\hline
\end{array}
\]

\begin{lemma}
For any $k\ge 3$, and for any idempotent quasigroup $(K,\circ )$ of order $k$, $U_k$ is a
non-separable $k$-latin square.
\end{lemma}

\begin{proof}
We leave the proof that $U_k$ is a $k$-latin square as an exercise.

To prove that $U_k$ is non-separable, suppose that $U_k$ contains an $l$-latin
square $S$ of order $k+2$, $l\ge 1$.
Then we have $x\in S(k+1,k+1)$, for some $x\in K$.

The remaining $k-1$ copies of the symbol $x$ in row $k+1$ of $U_k$ occur in cell $(k+1,x)$,
and exactly $l-1$ of these copies must occur in $S$, so $S(k+1,x)=\{x^{l-1},k+2\}$.
To obtain $l$ copies of symbol $x$ in column $x$, we must have $S(x,x)=\{x,(k+1)^{l-1}\}$.
The only other occurrence of symbol $k+1$ in column $x$ of $U_k$ is in cell 
$(x\odot  1,x)$, so we have $k+1 \in S(x\odot  1, x) $. 
The remaining $k-1$ copies of symbol $k+1$ in row $x\odot  1$ of $U_k$ occur in cell 
$(x\odot  1,x\odot  1)$, so $S(x\odot  1,x\odot  1) = \{(k+1)^{l-1},x\odot 1\}$.
In the same way we have $S(k+1,x\odot  1)=\{(x \odot 1)^{l-1},k+2 \}$, and finally
$x\odot  1 \in S(k+1,k+1)$.

It follows inductively that $K \subseteq S(k+1,k+1)$, and hence that $l=k$.
\end{proof}
\vspace{3mm}

We now modify the construction for $U_k$ to obtain non-separable $k$-latin squares
of every order between $k+3$ and $2k+1$. Let $s\in \{1,2,\dots, k-1\}$ and define
$K'=N(s+2)$.
We use a second idempotent quasigroup $(K', \star )$ of order $s+2$.
Define $U_{k,s}$ to be a $k$-latin square of order $k+s+2$ with
$$
\begin{array}{l}
U_{k,s}(i,i) = \{i,(k+1)^{k-1}\}, \;\; i\in K, \\
U_{k,s}( i\odot 1 ,i) 
	= \{k+1, ((i\odot 1 )\circ i)^{k-s-1} ,k+3,k+4,\dots ,k+s+2\}, \;\; i\in K, \\
U_{k,s}(i,j) 
= \{k+2, (i\circ j)^{k-s-1},k+3,k+4,\dots ,k+s+2\}, \;\; i,j\in K, \; i\ne j, j\odot  1, \\
U_{k,s}(i,k+1) = U_{k,s}(k+1,i) = \{i^{k-1},k+2\}, \;\; i\in K, \\
U_{k,s}(i,k+x) = U_{k,s}(k+x,i) = \{k+x\}\cup K \setminus \{i\}, \;\; i\in K, \; 2\le x\le s+2, \\
U_{k,s}(k+x,k+y) = K, \;\; x,y \in K', \; x\star y = 1, \\ 
U_{k,s}(k+x,k+y) = (k+1)^k, \;\; x,y \in K', \; x\star y = 2, \\ 
U_{k,s}(k+x,k+x) = (k+2)^k \;\; x \in K', \; x\ge 3, \\
U_{k,s}(k+x,k+y) = (k+z)^k, \;\; x,y,z \in K', \; x\star y = z \ge3, \; x\ne y.\\ 
\end{array}
$$

We illustrate with an example where $k=4$ and $s=2$. We use the same idempotent 
quasigroup $(K,\circ)$ as above. 
Since $s+2=k$ we let $(K',\star )$ be the same quasigroup also.

We obtain
$U_{4,2}=$

\[
\begin{array}{|c|c|c|c||c|c|c|c|}
\hline
5,5,5,1 & 4,7,8,6 & 2,7,8,6 & 3,7,8,5 & 1,1,1,6 & 2,3,4,6 & 2,3,4,7 & 2,3,4,8  \\
\hline
3,7,8,5 & 5,5,5,2 & 4,7,8,6 & 1,7,8,6 & 2,2,2,6 & 1,3,4,6 & 1,3,4,7 & 1,3,4,8 \\
\hline
4,7,8,6 & 1,7,8,5 & 5,5,5,3 & 2,7,8,6 & 3,3,3,6 & 1,2,4,6 & 1,2,4,7 & 1,2,4,8 \\
\hline
2,7,8,6 & 3,7,8,6 & 1,7,8,5 & 5,5,5,4 & 4,4,4,6 & 1,2,3,6 & 1,2,3,7 & 1,2,3,8 \\
\hline \hline
1,1,1,6 & 2,2,2,6 & 3,3,3,6 & 4,4,4,6 & 1,2,3,4 & 8,8,8,8 & 5,5,5,5 & 7,7,7,7 \\
\hline
2,3,4,6 & 1,3,4,6 & 1,2,4,6 & 1,2,3,6 & 7,7,7,7 & 5,5,5,5 & 8,8,8,8 & 1,2,3,4 \\
\hline
2,3,4,7 & 1,3,4,7 & 1,2,4,7 & 1,2,3,7 & 8,8,8,8 & 1,2,3,4 & 6,6,6,6 & 5,5,5,5 \\
\hline
2,3,4,8 & 1,3,4,8 & 1,2,4,8 & 1,2,3,8 & 5,5,5,5 & 7,7,7,7 & 1,2,3,4 & 6,6,6,6 \\
\hline
\end{array}
\]

Again we leave the proof that $U_{k,s}$ is a $k$-latin square as an excercise.
The proof that $U_{k,s}$ is non-separable is exactly the same as the equivalent proof 
for $U_k$, because for any $x\in K$ the occurrences of the symbols $x$ and $k+1$ in row
$k+1$ and column $x$ are identical in $U_k$ and $U_{k,s}$. Thus we have the following:

\begin{corollary}
For any $k\ge 3$ and $s\in \{1,2,\dots , k-1\}$, and for any idempotent 
quasigroups $(K,\circ )$ of order $k$ and $(K',\star )$ of order $s+2$, $U_{k,s}$ is a
non-separable $k$-latin square of order $k+s+2$.
\end{corollary}

Using the above corollary and Theorem \ref{embed1} from 
the previous section, we obtain the following theorem:

\begin{theorem}
For any integers $k\ge 3$ and $n\ge k+2$, there exists a non-separable $k$-latin square 
of order $n$.
\end{theorem}

Even for small values of $k$ and $n$ the above result is not the best
possible.  
For order $4$, we exhibit non-separable $k$-latin
squares for each $k\in \{3,4,5,6\}$.  These examples were found
by computer and checked by hand.

\[\begin{array}{|c|c|c|c|}
\hline
4,3,3 & 2,2,2 & 4,1,3 &  4,1,1 \\
\hline
1,1,3 & 4,1,3 & 2,2,2 &  4,4,3 \\
\hline
2,2,2 & 4,1,1 & 4,4,3 & 1,3,3 \\
\hline
4,4,1 & 4,3,3 & 1,1,3 & 2,2,2 \\
\hline
\end{array}
\quad
\begin{array}{|c|c|c|c|}
\hline
4,2,3,3 & 4,2,2,2 & 1,1,3,3 &  4,4,1,1 \\
\hline
4,1,2,3 & 1,1,2,3 & 4,2,3,3 &  4,4,1,2 \\
\hline
1,1,1,2 & 4,4,4,3 & 4,2,2,2 & 1,3,3,3 \\
\hline
4,4,2,3 & 1,1,3,3 & 4,4,1,1 & 2,2,2,3 \\
\hline
\end{array}\]
\[\begin{array}{|c|c|c|c|}
\hline
4,2,2,3,3 & 4,1,1,1,1 & 1,2,2,2,3 &  4,4,4,3,3 \\
\hline
1,2,3,3,3 & 4,4,2,2,2 & 4,4,4,3,3 &  1,1,1,2,2 \\
\hline
4,4,4,2,2 & 1,2,2,3,3 & 4,4,1,1,1 & 1,2,3,3,3 \\
\hline
4,1,1,1,1 & 4,4,3,3,3 & 1,2,2,3,3 & 4,4,2,2,2 \\
\hline
\end{array}\]
\[\begin{array}{|c|c|c|c|}
\hline
4,3,3,3,3,3 & 2,2,2,2,2,3 & 1,1,1,1,1,1 &  4,4,4,4,4,2 \\
\hline
1,1,2,2,2,3 & 4,1,1,1,1,2 & 4,4,4,4,4,3 &  2,2,3,3,3,3 \\
\hline
1,1,1,2,2,2 & 4,4,4,4,4,3 & 2,2,2,3,3,3 & 4,1,1,1,3,3 \\
\hline
4,4,4,4,4,1 & 1,1,3,3,3,3 & 4,2,2,2,3,3 & 1,1,1,2,2,2 \\
\hline
\end{array}\]

We were unable to find a non-separable $7$-latin square of order $4$. 
It seems plausible that such a configuration does not exist; however
we were unable to check every possible case by computer. 

We show the following.
\begin{theorem}
For each positive integer $n$, 
there exists some finite value $g(n)$ such that for any $k\geq g(n)$, 
every $k$-latin square of order $n$ is separable.
\label{conj2}
\end{theorem}

\begin{proof}
Suppose that $g(n)$ does not exist for some fixed $n$. Then there exists an
infinite sequence $(L_t)$ of non-separable $k(t)$-latin squares of order $n$, 
where $k(t)$ is strictly increasing. 
For each $i,j,s\in N(n)$, let $m_t(i,j,s)$ be the number of copies of $s$ in 
cell $L_t(i,j)$. 
Since our sequence has infinite length, for fixed $(i,j,s)\in N(n)\times N(n)\times N(n)$, the sequence $(m_t(i,j,s))$ contains a non-decreasing sub-sequence of infinite length. 
Next, replace $(L_t)$ with one of its infinite subsequences so that $(m_t(i,j,s))$ is non-decreasing.  
We repeat this process for each $(i,j,s)\in N(n)\times N(n)\times N(n)$, obtaining an infinite subsequence $(R_r)$. 
Since $k(t)$ is strictly increasing, 
all but possibly the first multi-latin square in $(R_r)$ is separable. This is  a contradiction, so our theorem is true.
\end{proof}

It is an open problem to determine $g(n)$ exactly for each $n\geq 1$. 
It is an easy exercise to show that $g(1)=g(2)=2$. 
We conjecture that $g(3)=3$ (the data in Table 1 in the next section certainly supports this)  
and that $g(4)=7$. 

\section{Computation}\label{sec5}

We remind the reader of the definition of {\em main class} and {\em paratopy} given in Section \ref{sec2}.

We wrote C++ code~\cite{klatincode} to enumerate $k$-latin squares using
the method of
{\em canonical augmentation} \cite{classificationAlgorithms},
\cite{McKayCanonical}. The main requirement is a function
$\mathcal{C}(K)$ that gives a {\em canonical label} of a (partial)
$k$-latin square. We require that 
$\mathcal{C}(K) = \mathcal{C}(K')$ if and only if $K$ and $K'$ are
paratopic. For this we generalise a well-known graph representation of a
latin square \cite{McMeMy}.
For a $k$-latin square $K$ of order $n$
we form a graph $G_K$ with vertex set
\begin{equation}\label{eqnVertices}
\{ v_1, \dots, v_{kn^2} \}
\cup
\{ r_1, \dots, r_{n},
c_1, \dots, c_{n},
s_1, \dots, s_{n} \}
\cup \{ R, C, S \}.
\end{equation}
Without loss of generality we may order the elements that appear in each
cell of $K$, so we may speak of the $y$th element in a cell, for 
$1 \leq y \leq k$. If $e$ is the $y$th symbol in cell $(i,j)$ of $K$, then
$G_K$ has the edges
$(v_{\ell}, r_i)$, 
$(v_{\ell}, c_j)$, and
$(v_{\ell}, s_e)$
where $\ell = e(n(i-1) + j-1) + y$.
Further, $G_K$ has the edges
$(r_i,R)$, $(c_i,C)$, $(s_i,S)$
for $1 \leq i \leq n$.
Further, we colour the vertices of 
$G_K$ according to the partitioning in (\ref{eqnVertices}).
The \texttt{nauty}~\cite{nauty} package provides a {\em canonical label} 
$\mathcal{C}(G_K)$ such that
$\mathcal{C}(G_K) = \mathcal{C}(G_{K'})$ if and only if
$G_K^{\gamma} = G_{K'}$ for some colour-preserving permutation $\gamma$
of the vertex labels of $G_K$. 
Since $\gamma$ preserves adjacencies and
colour classes, the following lemma easily follows.

\begin{lemma}
Let $K$ and $K'$ be two $k$-latin squares of order $n$. Then 
$\mathcal{C}(G_K) = \mathcal{C}(G_{K'})$ if and only if
$K$ and $K'$ are paratopic.
\end{lemma} 

We now take our canonical label to be 
$\mathcal{C}(K) = \mathcal{C}(G_K)$. 
We begin with a partial $k$-latin square of order $n$ with a
single row filled in, and then proceed by adding one row at a time.
For small $n$, $k$ Table 1 shows the number of main classes of 
$k$-latin squares of order $n$, and how many of these are
erodable, separable and simple.

\begin{center}
\begin{tabular}{|r|r|r|r|r|r|}
\hline
$n$ & $k$ & main classes & erodable & separable & simple \\
\hline
3 & 1 & 1 & 0 & 0 & 1 \\
\hline
3 & 2 & 4 & 3 & 3 & 1 \\
\hline
3 & 3 & 9 & 9 & 9 & 1 \\
\hline
3 & 4 & 24 & 22 & 24 & 0 \\
\hline
3 & 5 & 50 & 50 & 50 & 0 \\
\hline
3 & 6 & 117 & 115 & 117 & 0 \\
\hline
3 & 7 & 237 & 237 & 237 & 0 \\
\hline
3 & 8 & 488 & 485 & 488 & 0 \\
\hline
3 & 9 & 924 & 924 & 924 & 0 \\ 
\hline
4 & 1 & 2 & 0 & 0 & 2 \\
\hline
4 & 2 & 44 & 26 & 26 & 10 \\
\hline
4 & 3 & 2424 & 2181 & 2181 & 2 \\
\hline
4 & 4 & 218632 & 212942 & 218198 & 1 \\
\hline
5 & 1 & 2 & 0 & 0 & 2  \\
\hline
6 & 1 & 12 & 0 & 0 & 12 \\
\hline
\end{tabular}

\end{center}

\vspace{1mm}

\centerline{Table 1}

  \let\oldthebibliography=\thebibliography
  \let\endoldthebibliography=\endthebibliography
  \renewenvironment{thebibliography}[1]{%
    \begin{oldthebibliography}{#1}%
      \setlength{\parskip}{0.4ex plus 0.1ex minus 0.1ex}%
      \setlength{\itemsep}{0.4ex plus 0.1ex minus 0.1ex}%
  }%
  {%
    \end{oldthebibliography}%
  }

\section{Applications}

In this section we briefly discuss some possible applications of multi-latin squares to the design of statistical experiments.

Suppose that we want to compare $n$ varieties of tomato, $n$ types of compost and $n$ watering schemes, and we have $n^2k$ plots to do this in. 
 If we use a multi-latin square, assigning the varieties to rows, composts to columns and watering schemes to symbols, then each variety occurs $k$ times with each compost, each variety occurs $k$ times with each watering scheme, and each compost occurs $k$ times with each watering scheme.  This is a good design if we can assume that there are no {\em interactions}, which means that the 
difference 
in performance
between two varieties of tomato does not depend on the type of compost or 
the watering scheme, and similarly for two types of compost or two 
watering schemes. Such a design would be called an orthogonal 
main-effects factorial design for three $n$-level treatment factors.

For practical purposes, this wouldn't be done for $k\geq n$, because then it 
would be better to ensure that all of the potential $n^3$ combinations 
occurred at least once.

For a second sort of design, instead of watering schemes, 
suppose that we are going to use $n$ glasshouses, with $nk$ chambers in each 
glasshouse.  Even if we are not interested in the differences between the 
glasshouses, using a multi-latin square with the varieties, composts and 
glasshouses assigned to the rows, columns and symbols in some order gives 
us a good orthogonal main-effects factorial design for two $n$-level 
treatment factors in $n$ blocks of size $nk$.

\vspace{0.5cm}

\noindent {\bf \large Acknowledgments}:
We wish to thank Rosemary Bailey for her comments on the possible application of multi-latin squares to the efficient design of statistical experiments, which we have paraphrased in Section 6.

\end{document}